\documentclass[a4paper,12pt]{article}
\usepackage{amsmath, amssymb, amsthm}
\usepackage[dvips,final]{epsfig}
\usepackage[dvips,final]{graphicx}
\usepackage{amsmath}
\usepackage{amsthm}
\usepackage{enumerate}
\usepackage{setspace}
\usepackage{amsfonts}
\usepackage{empheq}
\usepackage{amssymb}
\usepackage{hyperref}
\usepackage[english]{babel}

\usepackage{graphicx}
\usepackage[left=3cm,right=3cm,top=2cm,
bottom=2cm]{geometry}
\usepackage{color}
\definecolor{lightgray}{rgb}{0.83, 0.83, 0.83}
\definecolor{lightblue}{rgb}{0.83, 0.83, 1}
\definecolor{ao}{rgb}{0.0, 0.5, 0.0}
\definecolor{dg}{rgb}{0.01, 0.75, 0.24}

\theoremstyle{plain}
\newtheorem{thm}{Theorem}[section]

\newtheorem{lem}[thm]{Lemma}
\newtheorem{rmk}[thm]{Remark}

\theoremstyle{definition}

\theoremstyle{definition}

\newcommand{\norm}[1]{\left\lVert#1\right\rVert}
\newcommand{\fhi}{\varphi}
\def\th{\theta}
\newcommand{\ph}{\varphi}

\def\dd{\,\mathrm{d}}
\def\dive{\mathrm{\,div\,}}
\def\ddt{\dfrac{d}{dt}}

\def\R{\mathbb{R}}

\def\AA{\mathcal{A}}
\def\SS{\mathcal{S}}
\def\PP{\mathcal{P}}
\def\CC{\mathcal{C}}
\def\BB{\mathcal{B}}

\def\th{\theta}
\def \pn{{\partial_{\boldsymbol n}}}
\def \pt{{\partial_t}}
\def\d{{\rm d}}

\def\ddt{\dfrac{\d}{\d t}}
\def \pt{{\partial_t}}
\def \pn{{\partial_{\boldsymbol n}}}
\def \ph{{\varphi}}

\def\PP{{\mathcal P}}
\def\into{{\int_\Omega}}
\def \th{{\vartheta}}

\def\io{\int_{\Omega}}

\def\12{\dfrac{1}{2}}
\def\p{\partial}
\def\pt{\partial_t}

\numberwithin{equation}{section}

\begin{document}

\title{An Allen-Cahn tumor growth model with temperature}
\author {Stefania Gatti\thanks{University of Modena and Reggio Emilia, Dipartimento di Scienze Fisiche, Informatiche e Matematiche, via Campi 213/b, I-41125 Modena (Italy).\newline
E-mail: \textit{stefania.gatti@unimore.it}}\;\;\;
Erica Ipocoana\thanks{Freie Universität Berlin,
Department of Mathematics and Computer Science, Arnimallee 9, 14195  Berlin (Germany).\newline
E-mail: \textit{erica.ipocoana@fu-berlin.de}}\;\;\;
Alain Miranville \thanks{Henan Normal University, School of Mathematics and information Science, Xinxiang, Henan (China).}
\thanks{Université Le Havre Normandie, Laboratoire de Mathématiques Appliquées du Havre (LMAH), 25, rue Philippe Lebon
BP 1123, 76063 Le Havre cedex (France).\newline
E-mail: \textit{alain.miranville@univ-lehavre.fr}}
}

\maketitle

\begin{abstract}
In this paper, we propose a new non-isothermal Allen-Cahn (Ginzburg-Landau) model for tumor growth. After deriving it using a microforces approach, we study its well-posedness. In particular, we are able to prove the existence and uniqueness of a local and global-in-time solution to our PDE system.
\end{abstract}

 \medskip \\
\medskip
\noindent \textbf{MSC 2020:}

\section{Introduction}
An increasing number of mathematical models have been proposed to support the medical community in better understanding and treating cancer (see \cite{byrne} and the review papers \cite{F-survey, iran, Xsurv}).
Indeed, mathematical models might provide further insights in tumor growth, focusing
on the reality-matching mathematical assumptions on parameters and analysis of long-time behaviour. Among these, phase-field models have been investigated intensively. In this framework, the evolution of the tumor is regulated by an order parameter and driven by several biological mechanisms such as proliferation via nutrient consumption, apoptosis \cite{I, mrs19}, chemotaxis and
active transport of specific chemical species \cite{garcke}. Given the medical nature of the problem, models taking into account the effects of therapies have also been studied e.g. for brain tumors \cite{AIMS} and prostatic cancer \cite{CGLMRR-wp, CGLMRR-contr}.\\
In spite of the abundance of mathematical models in tumor growth, still an aspect has so far been neglected, namely,
the thermal effects on cancer.
Nonetheless, it is experimentally verified that both high and low temperatures can lead to partial or complete destruction of tumor cells.
Sometimes hyperthermia or hypothermia are an adiuvant to other therapies, others they are a last resort for untreatable deseases (see the review \cite{temp} for the former treatment and \cite{hypo1}, \cite{hypo2}, \cite{Seki}
 for the latter).
It is also noticed that they are quite difficult to administer without damage to the surrounding healthy tissues and  the overall  mechanism behind their efficacy is not well understood.

In this paper, we present and study the well-posedness of a non-isothermal Allen-Cahn model for tumor growth coupled with nutrient dynamics.
Here, we suppose the evolution of the tumor to be governed by the proliferation via nutrient consumption, the apoptosis and the thermal influence. Although we do not include thermal control, the temperature contribution to the tumor equation is a first attempt to render the fact that low temperatures slow down tumor growth.
We also underline that the tumor cells are assumed to die only by apoptosis, therefore we do not consider e.g. necrosis \cite{garcke-necr, WLFC}.
Our thermodynamically consistent model allows to include the temperature influence on tumor growth, achieving at the same time the global existence and uniqueness of {\it not too weak} solutions. Indeed, up to our knowledge, the only other non-isothermal phase-field model for tumor growth in the literature is a Cahn-Hilliard model recently studied in \cite{I} (see also \cite{EGS} and \cite{ERS} for similar non isothermal phase-field equations for fluids mixtures). In this previous paper, only the existence of weak entropy solutions was proven, due to the  chemical potential term in the right hand side of the temperature equation and therefore to the impossibility of gaining enough regularity to pass to the limit.
As a consequence, also the uniqueness of the solution could not be achieved. Besides, the Cahn-Hilliard framework precludes any maximum principles for the phase-field. Here, switching to an Allen-Cahn system, we are able to overcome the analogous difficulties exploiting a completely different approach, based on the decoupling of the system, as proposed in \cite{MSJEE, fterich}. This strategy will be made clear in Section \ref{sec-wp}. \\
We remark that for Allen-Cahn models for tumor growth, the literature is not as extensive as for Cahn-Hilliard models. In fact, we refer to the works \cite{ CGLMRR-wp, XVG16, XVG16-, XVG20} for models and well-posedness and to \cite{AFMW-contr, CGLMRR-contr} for optimal control analysis.\\
Our PDE system consists of three evolution equations, governing the phase-field, the absolute temperature and the nutrient concentration, respectively:
namely, given a bounded regular domain  $\Omega \subset \R^3$, we consider
\begin{align} \label{eqfhi}
\beta\fhi_t-\Delta\fhi+F'(\fhi)-\theta=(\PP\sigma-\AA)h(\fhi)
& \quad\text{in}\quad \Omega\times (0,T)
\vspace{2mm}\\ \label{eqtemp}
c_V\theta_t-\dive[\kappa(\theta)\nabla\theta]-\beta(\fhi_t)^2+\theta \fhi_t=0
& \quad\text{in}\quad \Omega\times (0,T)\vspace{2mm}\\ \label{eqnutr}
\sigma_t-\Delta\sigma=-\CC\sigma h(\fhi)+\BB(\sigma_B-\sigma) & \quad\text{in}\quad \Omega\times (0,T)\vspace{2mm}\\ \label{bc}
\pn \fhi=\pn \theta=\pn\sigma=0 & \quad\text{on}\quad \partial\Omega\times (0,T)\vspace{2mm}\\ \label{ic}
\fhi(0)=\fhi_0,\quad \theta(0)=\theta_0,\quad\sigma(0)=\sigma_0 &\quad\text{in}\quad \Omega.
\end{align}
The order parameter or phase-field $\fhi$ denotes the local concentration of tumor cells: therefore, biologically,  $\ph$ should vary in a bounded interval, say $[0,1]$. It means that $\ph=+1$ in the regions exclusively occupied by tumor cells, $\ph=0$ in the healthy ones while $\ph\in (0,1)$ in the so-called diffuse interface, that is, a thin layer where the phases coexist.
Although we are able to analytically prove that $\ph\geq 0$ is bounded on any finite time interval, the general source term prevents to show that $\ph\leq 1$. The  absolute temperature is denoted by $\theta$  and $\sigma$ represents the concentration of a nutrient, such as oxygen and
glucose, for tumor cells: due to their meaning, both variables should be nonnegative. The parameters $\PP, \AA, \CC$, $\BB$ and $\sigma_B$ are positive constants indicating  tumor proliferation rate, apoptosis rate, nutrient consumption rate, the
blood-tissue transfer rate and nutrient concentration in the pre-existing vasculature, respectively. The function $h$, whose choice is detailed in Section \ref{asssec}, is a bounded, monotone increasing function supposed to be nonnegative on $[0,1]$. Therefore, assuming $\ph\geq 0$, the tumor grows when $(\PP\sigma-\AA)>0$ and reduces otherwise (in particular, the higher the concentration of tumor cells, the faster they die).
The term $\CC\sigma h(\fhi)$ represents the consumption of the nutrient by the tumor
cells. We here consider the case where the tumor has its own vasculature, which leads to the presence of the term $\BB(\sigma_B-\sigma)$ in the nutrient equation: in particular, if $\sigma_B>\sigma$, the term $\BB(\sigma_B-\sigma)$ models the supply of nutrient
from the blood vessels, otherwise,
the transport of nutrient away from the domain.\\
Eventually, $\kappa(\theta)$ stand for the heat conductivity, $c_V$ the specific heat and the nonnegative constant $\beta$ is a constitutive modulus.
Our last remark concerns the function $F(\fhi)$, supposed to be a polynomial double-well potential.
 \\

The plan of the paper is the following. As we here present a new phase-field model, we first derive it in Section \ref{secder}, following the microforces approach. Then, in Section \ref{sec-wp}, we investigate the well-posedness of a simplified version of the model. Namely, having introduced our assumptions and notation in
Section \ref{asssec}, we prove the existence and uniqueness of a local-in-time solution in Theorem \ref{localthm}, in Section \ref{seclocal}. Heuristically, the main idea behind the proof of this result is decoupling our problem, first  fixing $\theta$ in a suitable regularity space and then the couple $(\fhi,\sigma)$, in the spirit of \cite{MSJEE}. This yields to two intermediate results for the frozen problems. Each of these results, i.e. Lemmas \ref{lemmino1} and \ref{lemmino2}, relies on the Galerkin discretization method. The proof is concluded by composing the two frozen problems and using Schauder's fixed point theorem. Finally, in Section \ref{secglobal}, we show that we are able to extend the local solutions and thus we gain a global-in-time result, namely Theorem \ref{globalthm}.

\section{Derivation of the model}\label{secder}

We suppose that a two-component mixture consisting of healthy cells and tumor cells occupies an open spatial domain $\Omega \subset \R^3$. We denote by $\fhi(x,t)$ the tumor phase concentration, $\theta(x,t)$ the absolute temperature and $\sigma(x,t)$  the concentration of a nutrient for the tumor cells.\\
We follow Gurtin’s approach based on microforces (proposed in  \cite{gurtin} and largely used later on, e.g. in \cite{fterich, I, marvschimp, ms2005mod}.) in order to derive our model \eqref{eqfhi}--\eqref{eqnutr}. As a matter of fact, we treat separately the balance laws and the constitutive relations. In particular, we postulate the following balance law for internal microforces
\begin{eqnarray}\label{micro1}
\dive{\pmb{\xi}}+\pi + \gamma=0,
\end{eqnarray}
where $\pmb{\xi}$ is a vector representing the microstress, $\pi$ is a scalar corresponding to the internal microforces and $\gamma$ to external microforces. Microforces describe the forces associated with microscopic configurations of atoms, differently from standard forces, which are associated with macroscopic length scales. These different length scales are the reason why a separate balance law for microforces is needed. Such interatomic forces may be mirrored on the macroscopic level
by fields which perform work when the order parameter undergoes changes. Therefore this working can be described in terms of $\fhi_t$, which explains why the microforces are scalar rather than vector quantities.\\
For sake of completeness, we also refer to the pioneering paper \cite{BFL} after which the models of phase transition with microscopic movements developed by Fr\'{e}mond have been studied extensively. In particular, we cite \cite{LSS, LSS2002}, where the nonlinearities are similar to the ones appearing in our model.

\subsection{Order parameter equation}
The internal energy density of the system is given by Gibbs' relation
\begin{eqnarray}\label{inten}
e = \psi + \theta s,
\end{eqnarray}
where $s$ represents the entropy of the system, namely
\begin{eqnarray}\label{entropy}
s = -\dfrac{\partial\psi}{\partial\theta}.
\end{eqnarray}
The term $\psi=\psi(\theta,\ph,\nabla \ph)$ denotes the free energy, whose form will be chosen later in order to make this first part of the derivation more general.\\
Following \cite{gurtin}, we write the first law of thermodynamics in the form
\begin{eqnarray}\label{1tif}
\dfrac{d}{dt}\int_{\textit{R}} e \dd x = -\int_{\partial \textit{R}} \pmb{q}\cdot \pmb{\nu} \dd \Sigma + \mathcal{W}(\textit{R})+\mathcal{M}(\textit{R}),
\end{eqnarray}
where $\pmb{q}$ is the heat flux, \textit{R} is the control volume and $\pmb\nu$ is the outward unit normal to $\partial R$. Besides,
\begin{align}\label{W}
&\mathcal{W}(\textit{R}) = \int_{\partial \textit{R}} (\pmb{\xi} \cdot \pmb{\nu})\dfrac{\partial \fhi}{\partial t} \dd \Sigma + \int_R \gamma \dfrac{\partial \fhi}{\partial t} \dd x,\\ \label{M}
&\mathcal{M}(\textit{R}) =0
\end{align}
are the rate of working and the rate at which free energy is added to \textit{R}, under the  hypothesis of no heat supply, respectively. We recall that \eqref{M} holds since we here consider the case of the Ginzburg-Landau equation.\\
By Green's formula, since the control volume \textit{R} is arbitrary, \eqref{1tif} reads
\begin{eqnarray}\label{1tdf}
\dfrac{\partial e}{\partial t}= -\dive \pmb{q} + \dfrac{\partial\fhi}{\partial t} \dive \pmb\xi + \pmb\xi\cdot \nabla\dfrac{\partial \fhi}{\partial t} + \dfrac{\partial \fhi}{\partial t} \gamma.
\end{eqnarray}
Hence, from the microforce balance \eqref{micro1}, we infer

\begin{eqnarray}\label{firstlaw}
\dfrac{\partial e}{\partial t}= -\dive \pmb{q} -\pi \dfrac{\partial\fhi}{\partial t}+ \pmb\xi\cdot \nabla\dfrac{\partial \fhi}{\partial t}.
\end{eqnarray}

\noindent We impose the validity of the second law of thermodynamics in the form of the Clausius-Duhem inequality, namely
\begin{eqnarray}\label{C-D}
\theta \left( \dfrac{\partial s}{\partial t} + \dive \pmb{Q}\right) \geq 0,
\end{eqnarray}
where the entropy flux $\pmb{Q}$ is
\begin{equation}\label{qQ}
\pmb{Q}= \dfrac{\pmb{q}}{\theta}.
\end{equation}

\noindent Working on the left hand side of \eqref{C-D}, we have
\begin{align}\nonumber
\theta \left( \dfrac{\partial s}{\partial t} + \dive \pmb{Q}\right) &\stackrel{\eqref{inten}}{=}
\dfrac{\partial e }{\partial t}-\dfrac{\partial \psi}{\partial t}-s\dfrac{\partial \theta}{\partial t} +\theta\dive \pmb{Q}\\ \nonumber
&\stackrel{\eqref{qQ}}{=}\dfrac{\p e}{\p t} -\dfrac{\p \psi}{\p t}-s\dfrac{\p \theta}{\p t} + \dive \pmb{q} -\pmb{Q}\cdot\nabla \theta\\ \label{cd1}
& \stackrel{\eqref{firstlaw}}{=}-\pi\dfrac{\p \fhi}{\p t}+\pmb{\xi} \cdot\nabla \dfrac{\p\fhi}{\p t} -\dfrac{\p\psi}{\p t}-s\dfrac{\p \theta}{\p t}-\pmb{Q}\cdot\nabla\theta.
\end{align}

The third  and fourth terms in the last equation, computed by the chain rule and \eqref{entropy}, read
\begin{equation}\label{etrick}
\dfrac{\p \psi}{\p t}+s\dfrac{\p \theta}{\p t}
= \dfrac{\p \psi}{\p \fhi}\dfrac{\p \fhi}{\p t}+\dfrac{\p \psi}{\p \nabla \fhi}\dfrac{\p \nabla \fhi}{\p t}.
\end{equation}
Then, inserting  \eqref{etrick} in \eqref{cd1} implies

\begin{align*}
\theta \left( \dfrac{\partial s}{\partial t} + \dive \pmb{Q}\right)
& \stackrel{\;}{=}-\pi\dfrac{\p \fhi}{\p t}+\pmb{\xi} \cdot\nabla \dfrac{\p\fhi}{\p t}
-\dfrac{\p \psi}{\p \fhi}\dfrac{\p \fhi}{\p t}-\dfrac{\p \psi}{\p \nabla \fhi}\dfrac{\p \nabla \fhi}{\p t}
-\pmb{Q}\cdot\nabla\theta\\
& \stackrel{\;}{=}\left(-\pi-\dfrac{\p\psi}{\p\fhi}\right)\dfrac{\p \fhi}{\p t}+\left(\pmb{\xi}- \dfrac{\p \psi}{\p \nabla \fhi}\right)\dfrac{\p\nabla\fhi}{\p t} -
\pmb{Q}\cdot\nabla\theta.
\end{align*}
Hence, in order for \eqref{C-D} to hold, we take
\begin{align}\label{xi}
&\pmb{\xi}= \dfrac{\p \psi}{\p \nabla \fhi},\\ \label{pineq}
&\left(-\pi-\dfrac{\p\psi}{\p\fhi}\right)\dfrac{\p \fhi}{\p t} -
\pmb{Q}\cdot\nabla\theta \geq 0.
\end{align}
Dividing \eqref{pineq} by $\theta$ (supposed to be strictly greater that zero), relation \eqref{qQ} yields to
\begin{equation}\label{ineqcost}
\dfrac{1}{\theta}\left(\pi+\dfrac{\p\psi}{\p\fhi}\right)\dfrac{\p \fhi}{\p t}+\dfrac{\pmb{q}}{\theta^2}\cdot\nabla\theta=
\dfrac{1}{\theta}\left(\pi+\dfrac{\p\psi}{\p\fhi}\right)\dfrac{\p \fhi}{\p t}-\pmb{q}\cdot\nabla\dfrac 1\theta  \leq 0.
\end{equation}
Arguing similarly as in \cite{ms2005mod}, we may take
\begin{equation}\label{qchoice}
\pmb{q}=\pmb{b}\dfrac{\p\fhi}{\p t}+\mathbb{B}\nabla\dfrac 1\theta,
\end{equation}
where $\pmb{b}$ is a vector and $\mathbb{B}$ is a matrix possibly depending on the constitutive variables $\fhi,\nabla\fhi,\theta,\nabla\theta,\sigma.$ According to \cite{gurtin}, we set
\begin{equation}\label{betaconst}
\dfrac{1}{\theta}\left(\pi+\dfrac{\p\psi}{\p\fhi}\right)= -\beta\dfrac{\p \fhi}{\p t}-\pmb{a}\cdot\nabla\dfrac 1\theta,
\end{equation}
where $\beta$ is a positive constant and $\pmb{a}$ is a vector depending on the constitutive variables. In particular, we set
\begin{align*}
\pmb{a}&=0,\\
\pmb{b}&=0,\\
\mathbb{B}&={\mathbb{I}_d\kappa(\theta)\theta^2},
\end{align*}
with
$\kappa(\theta)=1+\kappa_0\theta^q$, for some $q\geq 2$ and some $\kappa_0\geq 0$. These constitutive choices ensure that \eqref{ineqcost} holds.
Combining \eqref{betaconst} with the microforce balance \eqref{micro1}, it follows that
\begin{equation}\label{eqfinor}
\beta\dfrac{\p \fhi}{\p t}=\dfrac{1}{\theta}\dive \pmb{\xi}+\dfrac{1}{\theta}\gamma-\dfrac{1}{\theta}
\dfrac{\p\psi}{\p\fhi}.
\end{equation}

\noindent We now postulate the free energy density $\psi$ in the form
\begin{eqnarray}\label{fren}
\psi = \dfrac{\varepsilon}{2}|\nabla\fhi|^2 + \dfrac{1}{\varepsilon}F(\fhi)+f(\theta)-\theta \fhi + N(\fhi),
\end{eqnarray}
where $\varepsilon$ is a positive constant related to the thickness of the interface and the function $F$ denotes a polynomial double well potential having at least cubic growth at infinity. Since we assume $0$ and $1$ to be the pure phases, we will take $F(\ph)=\ph^2(1-\ph)^2$ (cf. \cite{XVG16}).
)\\
The term $f$ in \eqref{fren} describes the part of free energy which is purely caloric and is related to the specific heat $c_V(\theta) = Q'(\theta)$ through relation $Q(\theta)=f(\theta)-\theta f'(\theta)$.
Finally, $N$ represents both the chemical energy of the nutrient and the energy contributions given by the interactions between the tumor tissues and the nutrient.\\

\noindent Since $\psi$ is now supposed to take the form \eqref{fren}, according to \eqref{xi}, we infer
\begin{align}\label{xicost}
\pmb{\xi}=\varepsilon\nabla\fhi
\end{align}
and moreover
\begin{equation*}
\dfrac{\p\psi}{\p\fhi}= \dfrac{1}{\varepsilon}F'(\fhi)-\theta +\dfrac{\p N}{\p \fhi}.
\end{equation*}
The term $\frac{\p N}{\p \fhi}$ is known in the literature for describing chemotaxis (see e.g. \cite{garcke}), that we here neglect, therefore we take  $\frac{\p N}{\p \fhi}=0$.
Inserting these last relations in equation \eqref{eqfinor}, we obtain
\begin{equation*}
\beta\dfrac{\p \fhi}{\p t}=\dfrac{\varepsilon}{\theta}\Delta \fhi+\dfrac{1}{\theta}\gamma-\dfrac{1}{\theta\varepsilon}
F'(\fhi)+1.
\end{equation*}
Supposing that $\beta\theta$ is a constant still denoted by $\beta$, we get
\begin{equation}\label{eqfhigamma}
\beta\varphi_t-\varepsilon\Delta\fhi+\dfrac{1}{\varepsilon}F'(\fhi)-\theta = \gamma.
\end{equation}

\noindent Postulating the source term $\gamma$ to be
$$\gamma = (\PP\sigma-\AA)h(\fhi),$$
we recover equation \eqref{eqfhi}.
In particular, $\PP>0$ is the tumor proliferation rate, $\AA>0$ is the apoptosis rate while
the function $h$ renders how the tumor regulates several phenomena: in this case, proliferation and apoptosis.
Therefore, $h(0)=0$ and is monotone increasing with an upper threshold.

Notice that the constitutive relation regarding the source term $\gamma$ complies with the fact that tumor growth is proportional to the nutrient supply in the tumoral region. In particular, the tumor grows if $\PP\sigma-\AA>0$, otherwise it shrinks, according to apoptosis of tumor cells.

\subsection{Temperature equation}
According to \eqref{fren}, Gibbs' relation \eqref{inten} reads
\begin{align*}
e = \dfrac{\varepsilon}{2}|\nabla\fhi|^2+\dfrac{1}{\varepsilon}F(\fhi)+f(\theta)-\theta f'(\theta) + N(\fhi),
\end{align*}
therefore it holds

\begin{equation}\label{relen1}
\dfrac{\p e}{\p t}=\varepsilon\nabla\fhi\cdot\dfrac{\p\nabla\fhi}{\p t}+\dfrac{1}{\varepsilon}F'(\fhi)\fhi_t+Q'(\theta)\theta_t,
\end{equation}
since we recall that we set  $\frac{\p N}{\p \fhi}=0$. Moreover, combining the energy balance \eqref{firstlaw} with \eqref{qchoice}, \eqref{micro1} and \eqref{xi}, we get
\begin{align}\label{relen2}
\dfrac{\partial e}{\partial t}= -\dive\left(-\kappa(\theta)\nabla\theta-\dfrac{\p\psi}{\p\nabla\fhi}\dfrac{\partial\fhi}{\partial t}\right)+ \gamma \dfrac{\partial\fhi}{\partial t}.
\end{align}
Putting together \eqref{relen1} and \eqref{relen2} and recalling \eqref{xicost}, it follows
\begin{equation}\label{relen3}
Q'(\theta)\theta_t-\dive(\kappa(\theta)\nabla\theta)=\varepsilon\Delta\fhi\fhi_t-\dfrac{1}{\varepsilon}F'(\fhi)\fhi_t+\gamma\fhi_t.
\end{equation}
Substituting the term $\varepsilon\Delta\fhi$ in \eqref{relen3} through relation \eqref{eqfhigamma} and assuming a constant specific heat $c_V$, that is, $Q'(\theta)=c_V$ constant, we obtain the temperature equation \eqref{eqtemp}.

\subsection{Nutrient equation}
The nutrient balance equation is postulated in the form
\begin{eqnarray} \label{nutrbal}
\sigma_t = -\dive \mathbf{J}_\sigma-\mathcal{S},
\end{eqnarray}
where $J_\sigma$ is the nutrient flux and $\mathcal{S}$ denotes a source/sink term for the nutrient. Making the following constitutive choices, namely
\begin{align}
&\mathbf{J}_\sigma = -\nabla\sigma\\ \label{sourcenutr}
&\SS=\CC\sigma h(\fhi)-\BB(\sigma_B-\sigma),
\end{align}
we recover \eqref{eqnutr}.
{Here, we recall that $\CC>0$ corresponds to the nutrient consumption rate, $\BB>0$ is the nutrient supply rate while $\sigma_B$ is the vascular nutrient concentration. For the sake of simplicity, the nutrient diffusion constant has been normalized.}

Hence, the sink/source term in the nutrient equation is assumed to be regulated by both consumption of nutrients and nutrient supply via capillary network. In particular, we here consider the case where the nutrient supply is linked to the fact that the tumor has its own vasculature, according to the latter term in \eqref{sourcenutr}. Therefore, the term $\sigma_B$ acts as a threshold, indicating whether the nutrient is drawn to the tumor or transported in the other direction.


\section{Well-posedness}\label{sec-wp}
Having derived our model, we can now make precise our assumptions and subsequently our notation.

\subsection{Assumptions and notation}\label{asssec}
We consider the double well potential $F(\ph)=\ph^2(1-\ph)^2$ so that $F'(\ph)=4\ph^3-6\ph^2+2\ph$.
Concerning $h$, we assume $h\in C^1(\R)$ to be a monotone increasing function.
{We further suppose that $h(0)=0$ 
and
\begin{equation}
\label{acca}
|h(r)|+|h'(r)|\leq C,\quad r\in \R,
\end{equation}
for some positive constant $C$.
Notice that $h(0)=0$ complies with the biological meaning of the corresponding term in the tumor and nutrient equations where it renders
the tumor regulation of proliferation and apoptosis, and the nutrient uptaken by the tumor, respectively.\\

\noindent We suppose that the heat conductivity takes the form
\begin{equation}\label{kappa}
\kappa(r)=1+r^q,\quad r\geq 0,\quad \text{for some}\quad q\geq 2.
\end{equation}
This choice of $q$ mainly depends on mathematical reasons and is related to a better integrability for $\theta$.
Since in principle the sign of the variable $\theta$ is unknown, we extend $\kappa$ to the whole real line as $\kappa(r)=1+|r|^q$.
In the end, the modulus would be useless since the solution to the corresponding problem is nonnegative (actually $\theta\geq \underline\theta>0$, provided we properly modify the assumption  on the initial datum, as we explain in Remark \ref{segnormk}).
Besides, we suppose the coefficient $\sigma_B\in [0,1]$ and recall that
$\PP,\AA,\CC,\BB>0$ are constant as well.
Finally, for the sake of simplicity, we fix
$c_V=1$ and $\beta=\varepsilon=1$
in system \eqref{eqfhi}--\eqref{eqnutr}. Then the model reads

\begin{equation}
\label{system}
\begin{cases}
\ph_t-\Delta\ph+F'(\ph)-\theta=(\PP \sigma-\AA) h(\ph)
& \quad\text{in}\quad \Omega\times (0,T)
\vspace{2mm}\\
\theta_t-\dive[\kappa(\theta)\nabla\theta]-\ph^2_t+\theta \ph_t=0
& \quad\text{in}\quad \Omega\times (0,T)\vspace{2mm}\\
\sigma_t-\Delta\sigma=-\mathcal C\sigma h(\ph)+\mathcal B(\sigma_B-\sigma) & \quad\text{in}\quad \Omega\times (0,T)\vspace{2mm}\\
\end{cases}
\end{equation}
supplemented with \eqref{bc} and \eqref{ic}.
We recall that $\Omega$ is a smooth and bounded domain of $\R^3$ with boundary $\Gamma$, while
$(0,T)$ represents an assigned but otherwise arbitrary time interval. Then, we define the parabolic domain $Q_T:= \Omega \times (0,T)$.\\
We set $H:=L^2(\Omega)$ and $V:=H^1(\Omega)$, using these symbols also referring to vector valued functions.
In general, given a Banach space $X$, $\|\cdot\|_X$ represents its norm, using for brevity  $\|\cdot\|$  in place of $\|\cdot\|_H$, and
$\|\cdot\|_*$ for the norm of $V'$, dual space to $V$.
The symbol
$\langle \cdot,\cdot \rangle$ will stand for the duality pairing between $V'$ and $V$.
We also omit the variables of integrations when a misinterpretation may not occur.\\
Finally, $c$ stands for a positive constant, possibly changing from line to line: it is in general structural, further dependencies being pointed out on occurrence.

\smallskip\noindent
In order to make precise the assumptions on the initial data,
we define
\begin{equation} \label{kappone}
K(r) := \int_0^r \kappa(s) ds=r+\dfrac{r|r|^q}{q+1},  \quad r\in\mathbb{R},
\end{equation}
where $\kappa(r)=1+|r|^q,$
for some $q\geq 2,\ r\in\mathbb{R}$ (cf. \eqref{kappa}).
Indeed, we suppose that
\begin{align}
\label{initialtemp}
& \theta_0\in L^{3q}(\Omega)\cap V \quad \textrm{s.t.}\quad K(\theta_0)\in V \quad \textrm{and}\quad\theta_0\geq 0\quad \text{a.e. in}\quad \Omega,\\
\label{initialfhi}
& \ph_0\in W^{\frac 53,6}(\Omega)
,\quad \ph_0\geq 0\quad \text{a.e. in}\,\, \Omega, \quad \text{and}\quad \pn \fhi_0=0 \quad \text{on}\quad\partial\Omega\\
\label{initialnutr}
&\sigma_0\in L^\infty(\Omega),\quad 0\leq \sigma_0\leq 1 \quad \text{a.e. in}\quad \Omega.
\end{align}

\noindent
For further reference, we notice in particular that
\begin{equation}\label{K11}
\int_{\Omega} \kappa(\theta)^2 \vert \nabla \theta \vert^2 \, {\rm dx} = \norm{\nabla K(\theta)}^2,
\end{equation}
whenever it makes sense.
We also recall the nonlinear Poincar\'e inequality (see \cite{GMRS})
\begin{equation}
\label{nlPoi}
\||v|^{p/2}\|^2_V\leq c_p \left( \|v||^p_{L^1}+\|\nabla |v|^{p/2}\|^2\right),\quad  \forall v\in L^1(\Omega)\quad \text{s.t.}\quad \nabla |v|^{p/2}\in L^2(\Omega),
\end{equation}
holding for any finite $p\geq 2$.
Besides, without mentioning it, we will repeatedly use the following equation
\begin{equation}
\label{nablone}
\into |v|^r |\nabla v|^2=\dfrac 4{(r+2)^2}\into |\nabla ( |v|^{\frac r2+1})|^2.
\end{equation}

\subsection{Local in time existence of a solution}\label{seclocal}
We can now state the existence and uniqueness of a local solution that we will prove via Schauder fixed point theorem
in several steps. Namely,
\begin{thm}
\label{localthm}
For any $R>0$ there exists $T=T(R)>0$ such that if $(\theta_0,\ph_0,\sigma_0)$ satisfies assumptions \eqref{initialtemp}--\eqref{initialnutr}
and
\begin{equation}
\label{rbound}
\|\theta_0\|_{V\cap L^{3q}(\Omega)}+\|K(\theta_0)\|_V+\|\ph_0\|_{W^{\frac 53, 6}(\Omega)}\leq R,
\end{equation}
then there exists a unique local solution $(\theta,\ph,\sigma)$ to system \eqref{system} in $(0,T)$, namely,
\begin{align*}
&\theta\in C([0,T];H)\cap L^\infty(0,T;L^{3q}(\Omega)\cap V)\cap L^{4q}(0,T;L^{12q}(\Omega))\cap H^1(0,T;H)\\
&K(\theta)\in L^\infty(0,T;V)\cap L^2(0,T ;H^2(\Omega))\cap H^1(0,T;L^{\frac 32}(\Omega))\\
& \fhi\in C([0,T];V)\cap W^{1,4}(0,T;L^{6}(\Omega))\\
& \sigma \in L^\infty(0,T;L^\infty(\Omega))\cap L^2(0,T;V)\cap H^1(0,T;V')\\
& \theta_t -\Delta K(\theta) =\fhi_t^2 - \fhi_t\theta,\quad  \text{a.e. in}\quad  Q_T\\
&\ph_t-\Delta\ph+F'(\ph)=\theta+(\PP \sigma-\AA) h(\ph), \quad \textrm{a.e. in}\quad  Q_T\\
\label{eqfrnutr}
&\langle \sigma_t,v\rangle+\io  \nabla \sigma\cdot\nabla v+\mathcal B \io \sigma v=\io(-\mathcal C\sigma h(\ph)+\mathcal B \sigma_B)v, \quad \forall v\in V,\,\textrm{a.e. in}\quad (0,T)\\ \nonumber
&\pn \ph=0 \quad \textrm{and}\quad \pn K(\theta)=0, \quad \textrm{a.e. in }\quad \partial \Omega \times (0,T)\\ \nonumber
&\theta(0)=\theta_0,\quad \ph(0)=\ph_0\quad \sigma(0)=\sigma_0,\quad \textrm{a.e. in } \Omega.
\end{align*}

\noindent Besides,
$$
\theta\geq 0,\quad 0\leq   \sigma\leq 1 \quad \text{and}\quad \ph\geq 0\quad \text{a.e. in}\quad Q_T.
$$
\end{thm}
The strategy used to prove the result above is to decouple the problem, first freezing the temperature, then the order parameter and the nutrient concentration, with suitable regularities. Hence, we gain the well-posedness of the fixed point map.\\
For any fixed $R, T>0$, we introduce the space for the fixed point argument as
\begin{equation}
\label{tetaspace}
{\Theta}_R(T):=\lbrace
\tilde\theta\in L^4(0,T;L^6(\Omega)) \textrm{ s.t. }\tilde\theta\geq 0\, \textrm{ a.e. in }\, Q_T \,\textrm{and}\,
\|\tilde\theta\|_{L^4(0,T;L^6(\Omega)) }\leq R
\rbrace.
\end{equation}
\noindent The first step to prove Theorem \ref{localthm} is gaining the following result

\begin{lem}
\label{lemmino1}
Let assumptions \eqref{initialfhi} and \eqref{initialnutr} hold true. Given $R,T>0$, fix $\tilde\theta\in {\Theta}_R(T)$.
Then there exists a unique pair $(\ph,\sigma)$ such that
\begin{small}
\begin{align}\nonumber
&\ph\in C([0,T];V)\cap W^{1,4}(0,T;L^{6}(\Omega)),\quad \sigma \in L^\infty(0,T;L^\infty(\Omega))\cap L^2(0,T;V)\cap H^1(0,T;V')\\ \label{eqfrfhi}
&\ph_t-\Delta\ph+F'(\ph)=\tilde\theta+(\PP \sigma-\AA) h(\ph), \quad \textrm{a.e. in}\quad  Q_T\\ \label{eqfrnutr}
&\langle \sigma_t,v\rangle+\io \nabla \sigma\cdot \nabla v+\mathcal B\io \sigma v=\io(-\mathcal C\sigma h(\ph)+\mathcal B \sigma_B)v, \quad \forall v\in V,\,\textrm{a.e. in}\quad (0,T)\\ \nonumber
&\pn \ph=0, \quad \textrm{a.e. in }\quad \partial \Omega \times (0,T)\\ \nonumber
&\ph(0)=\ph_0\quad \sigma(0)=\sigma_0,\quad \textrm{a.e. in } \Omega.
\end{align}
\end{small}

\noindent Moreover,
$$
0\leq   \sigma\leq 1\quad \text{and}\quad \ph\geq 0\quad \text{a.e. in}\quad Q_T.
$$
In particular,
$$\|\ph\|_{L^\infty(0,T;L^\infty(\Omega))}+\|\fhi_t\|_{ L^4(0,T;L^6(\Omega))}\leq m_0(R) m_1(T),$$
where $m_i>0$ are monotone increasing functions.
\end{lem}

\begin{proof}
Assuming the existence of a pair $(\ph,\sigma)$ meeting the above requirement its uniqueness is straightforward.
We will argue formally, although the Galerkin method allows to make rigorous the next computations (see, e.g. \cite{AIMS}).
Indeed, testing \eqref{eqfrfhi} by $\ph$ and $-\Delta\ph$ in $H$ and \eqref{eqfrnutr} by $\sigma$ in $H$, owing to
\eqref{acca}, we easily see that
$$
\|\ph\|^2_{L^\infty(0,T;V)\cap L^2(0,T;H^2(\Omega))}+
\|\sigma\|^2_{L^\infty(0,T;H)\cap L^2(0,T;V)}\leq m_0(R)m_1(T),
$$
where, for $i=0,1$, $m_i$ is a positive and increasing function which in this proof may change from line to line.
Besides, by comparison in the equations, we further infer
$$
\|\ph_t\|^2_{L^2(0,T;H)}+\|\sigma_t\|^2_{L^2(0,T;V')}\leq m_0(R)m_1(T)
$$
Then the local (and actually global) existence of a pair $(\ph,\sigma)\in C([0,T],V)\times C([0,T],H)$
can be deduced by the Galerkin method.

Moreover, notice that $\sigma$ is a weak solution to
$$
\sigma_t-\Delta \sigma=\BB(\sigma_B-\sigma)-\CC \sigma h(\ph)\in L^\infty(0,T;H)
$$
with $\sigma_0\in L^\infty(\Omega)$. Applying classical regularity results \cite[Theorem 7.1 p. 181]{LSU},  we can see that $\sigma\in L^\infty(Q_T)$.
At this point, we can also prove that $\sigma\geq 0$ a.e. in $Q_T$.
Indeed, testing \eqref{eqfrnutr} by $-\sigma_{-}$, where $\sigma_{-}\geq 0$ is the negative part of $\sigma$, we infer
\begin{equation}
\dfrac{1}{2}\dfrac{d}{dt}\Vert \sigma_-\Vert^2+\Vert \nabla \sigma_-\Vert^2\leq c\Vert \sigma_-\Vert^2.
\end{equation}
By Gronwall's Lemma, the assumption $\sigma_0\geq 0$ entails $\sigma(x,t)\geq 0$.
Recalling $\sigma \in L^\infty(Q_T)$, we learn
$$
0\leq \sigma\leq \overline \sigma_T,\quad \text{a.e. in}\quad Q_T,
$$
for some $\overline\sigma_T\geq 1$.

\medskip\noindent
We now turn to the lower bound for the order parameter,  testing  the corresponding equation \eqref{eqfrfhi} by $-\fhi_-$, where $\fhi_{-}\geq 0$ is the negative part of $\fhi$.
Therefore we get
\begin{align}
&\dfrac{1}{2}\dfrac{d}{dt}\Vert \fhi_-\Vert^2+\Vert \nabla \fhi_-\Vert^2 +4\Vert \fhi_-\Vert_{L^4(\Omega)}^4+ 6\Vert \fhi_-\Vert_{L^3(\Omega)}^3+2\Vert \fhi_-\Vert^2\nonumber\\
&=-\int_\Omega\tilde{\theta}\fhi_- - \int_\Omega (\PP \sigma-\AA) h(\ph)\fhi_-.
\end{align}
Owing to $h(0)=0$ and the Lipschitz continuity in \eqref{acca}, we estimate the last term on the right hand side as
\begin{align*}
- \int_\Omega (\PP \sigma-\AA) h(\ph)\fhi_- \leq \|\PP \sigma-\AA\|_{L^\infty(\Omega)}\into| h(\ph)|\ph_-
\leq c \into |\ph|\ph_-=c\|\ph_-\|^2
\end{align*}
for some $c=c_T$ since we exploit $0\leq \sigma\leq \bar \sigma_T$ a.e. in $Q_T$.
Then, being by assumption $\tilde\theta\geq 0$, the differential inequality reduces to
\begin{align*}
\dfrac{1}{2}\dfrac{d}{dt}\Vert \fhi_-\Vert^2+\Vert \nabla \fhi_-\Vert^2 +4\Vert \fhi_-\Vert_{L^4(\Omega)}^4+ 6\Vert \fhi_-\Vert_{L^3(\Omega)}^3+2\Vert \fhi_-\Vert^2\leq c\|\ph_-\|^2 .
\end{align*}
Therefore, by Gronwall's Lemma and $\ph_0\geq 0$ a.e. in $\Omega$, we get that $\ph\geq 0$ a.e. in $Q_T$.\\

Going back to the nutrient's equation, we can now show that $\sigma\leq 1$ a.e. in $Q_T$. Indeed, having learned the sign of the order parameter, testing \eqref{eqfrnutr}
by $(\sigma-1)_+$ we have
\begin{align*}
& \dfrac 12 \ddt \|(\sigma-1)_+\|^2+\|\nabla(\sigma-1)_+\|^2+\mathcal B\|(\sigma-1)_+\|^2\\
&=
\mathcal B(\sigma_B-1)\into(\sigma-1)_+-\mathcal C\into h(\ph)\sigma (\sigma-1)_+\leq 0,
\end{align*}
since $\sigma_B\leq 1$. On account of the initial condition, we deduce $\sigma\leq 1$ a.e. in $Q_T$.

\medskip\noindent

An upper bound for the $L^\infty$-norm of the order  parameter $\fhi$ will also be available, as we prove by Moser's method.\\
Indeed, assumption \eqref{initialfhi}, in particular, provides
$$\fhi_0\in L^p(\Omega), \qquad p\in [1,+\infty].$$
Then, according to \eqref{eqfrfhi}, $\fhi$ solves
\begin{equation}\label{eqfhig}
\fhi_t-\Delta \fhi+4\fhi^3-6\fhi^2+2\fhi=g,
\end{equation}
where $g:= \tilde{\theta}+(\PP\sigma-\AA)h(\fhi)\in L^4(0,T;L^6(\Omega))$.
Testing \eqref{eqfhig} by $\fhi^p$, we obtain
\begin{align}\label{eqfhimoser}
\dfrac{1}{p+1}\dfrac{d}{dt}\int_\Omega \fhi^{p+1}-\io(\Delta\fhi)\fhi^p+4\io \fhi^{p+3}+2\io \fhi^{p+1}=6\io \fhi^{p+2}+\io g\fhi^p.
\end{align}
Since
\begin{align*}
-\io(\Delta\fhi)\fhi^p=\dfrac{4p}{(p+1)^2}\io \left|\nabla\left(\fhi^{\frac{p+1}{2}}\right)\right|^2,
\end{align*}
multiplying \eqref{eqfhimoser}  by $(p+1)$, we infer
\begin{align*}\nonumber
&\dfrac{d}{dt}\int_\Omega \fhi^{p+1}+
\dfrac{4p}{p+1}\io \left|\nabla\left(\fhi^{\frac{p+1}{2}}\right)\right|^2+4(p+1)\io\fhi^{p+3}+2(p+1)\io\fhi^{p+1}\\
&\quad=6(p+1)\io \fhi^{p+2}+(p+1)\io g\fhi^p.
\end{align*}
The second term on the right hand side might be read as
\begin{equation*}
(p+1)\io g\fhi^p= (p+1)\io g |\fhi|^\frac{p+1}{2}|\fhi|^\frac{p-1}{2}.
\end{equation*}
Then, using H\"older's inequality with exponents $6,6$ and $\frac32$ and the Sobolev embedding $W^{1,2}(\Omega)\hookrightarrow L^6(\Omega)$, we have
\begin{align*}
(p+1)\io g\fhi^p &\leq c(p+1)\Vert g\Vert_{L^6(\Omega)} \Vert \fhi^\frac{p+1}{2}\Vert_V\left(\io |\fhi|^{\frac{3}{4}(p-1)}\right)^\frac{2}{3}
\nonumber\\
&\leq
\dfrac{2p}{p+1}\Vert \fhi^\frac{p+1}{2}\Vert^2_V+c (p+1)^2\Vert g\Vert^2_{L^6(\Omega)}\left(\io |\fhi|^{\frac{3}{4}(p-1)}\right)^\frac{4}{3}\nonumber\\
& = \dfrac{2p}{p+1}\Vert \nabla(\fhi^\frac{p+1}{2})\Vert^2+\dfrac{2p}{p+1}\io  \ph^{p+1}
+c (p+1)^2\Vert g\Vert^2_{L^6(\Omega)}\left(\io |\fhi|^{\frac{3}{4}(p-1)}\right)^\frac{4}{3},
\end{align*}
where in the last line we have used once again Young's inequality. Here and throughout this proof, we emphasize that the positive constant $c$ is
independent of $p$.
Replacing these computations in the differential inequality, we obtain
\begin{align}\nonumber
&\dfrac{d}{dt}\int_\Omega \fhi^{p+1}+
\dfrac{2p}{p+1}\io \left|\nabla\left(\fhi^{\frac{p+1}{2}}\right)\right|^2+4(p+1)\io\fhi^{p+3}+\dfrac 2{p+1}\io \fhi^{p+1}\\
\label{ineqfhimoser}
&\quad\leq 6(p+1)\io \fhi^{p+2}
+c (p+1)^2\Vert g\Vert^2_{L^6(\Omega)}\left(\io |\fhi|^{\frac{3}{4}(p-1)}\right)^\frac{4}{3}.
\end{align}
Since the first term on the right hand side of \eqref{ineqfhimoser} can be estimated by the generalized Young's inequality as
\begin{align*}
6(p+1)\io \fhi^{p+2}\leq( p+1)\io \fhi^{p+3}+6^{p+3}(p+1) |\Omega|,
\end{align*}
we infer
\begin{align*}\nonumber
&\dfrac{d}{dt}\int_\Omega \fhi^{p+1}+
\dfrac{2p}{p+1}\io \left|\nabla\left(\fhi^{\frac{p+1}{2}}\right)\right|^2+3(p+1)\io\fhi^{p+3}\\
&\quad\leq c(p+1)6^{p+3}
+c (p+1)^2\Vert g\Vert^2_{L^6(\Omega)}\left(\io |\fhi|^{\frac{3}{4}(p+1)}\right)^\frac{4}{3}.
\end{align*}
Thus, integrating in time, it holds
\begin{align}
\label{moserfhitotal}
\int_\Omega \fhi^{p+1}(t)
&\leq \io \fhi_0^{p+1}+T
c(p+1)6^{p+3}
+c (p+1)^2\int_0^T\Vert g(s)\Vert^2_{L^6(\Omega)}\left(\io |\fhi(s)|^{\frac{3}{4}(p+1)}\right)^\frac{4}{3}ds
\end{align}
In order to apply Moser’s iteration scheme, let us consider the sequence $(p_k)_k$ of real numbers defined as
$$p_0=3, \qquad p_{k+1}=\dfrac{4}{3}p_k, \qquad k\in\mathbb{N}.$$
Choosing $p=p_{k+1}-1$ in \eqref{moserfhitotal}, we obtain
\begin{align*}
\int_\Omega \fhi^{p_{k+1}}(t)
\leq c\left(M^{p_{k+1}}+p^2_{k+1}+p^2_{k+1}\sup_{[0,T]}\left(\io\fhi^{p_k}\right)^{\frac{4}{3}}\right),
\end{align*}
for some constant $M$ depending on $T$, $\|\ph_0\|_{V\cap L^\infty(\Omega)}$ and $\|g\|_{L^2(0,T;L^6(\Omega))}$, but independent of $k$.
Hence, it follows
$$
\sup_{[0,T]}\int_\Omega \fhi^{p_{k+1}}(t)
\leq c p^2_{k+1} \max\left\lbrace M^{p_{k+1}}, \sup_{[0,T]}\left(\io\fhi^{p_k}\right)^{\frac{4}{3}}\right\rbrace,
$$
so we can conclude that $\ph\in L^\infty(0,T;L^\infty(\Omega))$ (see, e.g., \cite{Laurencot}) with
$$
\|\ph\|_{L^\infty(0,T;L^\infty(\Omega))}\leq m_0(R) m_1(T).
$$
Thus, having written \eqref{eqfrfhi}
as
$$
\ph_t-\Delta\ph=\tilde f,\quad \text{where}\quad \tilde f:=-F'(\ph)+\tilde\theta+(\PP \sigma-\AA) h(\ph)\in L^4(0,T;L^6(\Omega)),
$$
thanks to the regularity of the initial datum,
we deduce that $\ph_t,\Delta\ph\in  L^4(0,T;L^6(\Omega))$ (see the subsequent Remark) and

$$
\|\ph_t\|_{L^4(0,T;L^6(\Omega))}\leq c(\|\ph_0\|_{W^{\frac 53,6}(\Omega)}+\|\tilde f\|_{L^4(0,T;L^6(\Omega))})\leq m_0(R) m_1(T).
$$
\end{proof}
\begin{rmk}
The assumption \eqref{initialfhi}
on $\ph_0$ ensures the desired regularity on $\ph_t$ in Lemma \ref{lemmino1} as well as the related estimate that is crucial in the fixed point argument.
Indeed, assume that $(\ph,\sigma)$ is a solution to the problem corresponding to a fixed $\tilde\theta$: then, in particular,
$\ph$ solves
$$
\begin{cases}
\ph_t-\Delta\ph=\tilde f,& \quad \text{in}\quad \Omega\times (0,T)\\
\pn \ph=0,& \quad \text{on}\quad \Gamma\times (0,T)\\
\ph(0)=\ph_0,& \quad \text{in}\quad \Omega,
\end{cases}
$$
where $\tilde f=-F'(\ph)+\tilde\theta+(\PP \sigma-\AA) h(\ph)$.
On account of \cite[Theorem III.1]{vonwahl}, we observe that
$$
\ph_t\in L^4(0,T;L^6(\Omega)) \quad \text{and}\quad \|\ph_t-\psi_t\|_{L^4(0,T;L^6(\Omega))}\leq c(\|\ph_0\|_{L^6(\Omega)}+\|\tilde f\|_{L^4(0,T;L^6(\Omega))}),
$$
where $\psi$ is the solution to
$$
\begin{cases}
\psi_t-\Delta\psi=0,& \quad \text{in}\quad \Omega\times (0,T),\\
\pn \psi=0, & \quad \text{on}\quad \Gamma\times (0,T),\\
\psi(0)=\ph_0, & \quad \text{in}\quad \Omega.
\end{cases}
$$
Applying \cite[Theorem 9.1]{LSU} to the last system, we deduce that
$\psi\in W^{2,1}_{6}(Q_T)$ so that, in particular, $\psi_t\in L^{6}(Q_T)$
with $$\|\psi_t\|_{L^6(Q_T)}\leq c\|\ph_0\|_{W^{\frac 53,6}(\Omega)}.$$
Thus, we obtain the desired estimate
$$
\|\ph_t\|_{L^4(0,T;L^6(\Omega))}\leq c(\|\ph_0\|_{W^{\frac 53,6}(\Omega)}+\|\tilde f\|_{L^4(0,T;L^6(\Omega))}).
$$
\end{rmk}
%

%
%
%
%

\medskip\noindent
As a second step, we work on the temperature equation, freezing the order parameter and the nutrient concentration variables.\\ Therefore, given $R,T>0$, let us define the set

\begin{align*}
\mathcal X_R(T):=&\left\lbrace (\fhi, \sigma) \in [C([0,T];V)\cap W^{1,4}(0,T;L^{6}(\Omega))]\times L^\infty(0,T;L^\infty(\Omega)) ,\right.\\
&\left.\,\, \fhi\geq 0\, \text{and} \, 0\leq\sigma\leq 1 \,\textrm{ in } \, Q_T,\quad \|\ph_t\|_{L^4(0,T;L^6(\Omega))}\leq m_0(R)m_1(T)\right\rbrace.
\end{align*}
Notice in particular that for any $(\tilde \fhi,\tilde \sigma)\in \mathcal X_R(T)$, it holds  $\tilde{\ph}_t\in \mathcal H_\rho(T)$,
where
$$
\mathcal H_\rho(T):=\{ u\in L^4(0,T;L^6(\Omega)):\quad \|u\|_{L^4(0,T;L^6(\Omega))}\leq \rho\},
$$
being $\rho=m_0(R)m_1(T)$.

\medskip
\noindent Given $(\tilde \fhi,\tilde \sigma)\in \mathcal X_R(T)$, let $m:=\tilde{\ph}_t$ and study the problem
\begin{equation}
\theta_t-\dive[\kappa(\theta)\nabla\theta] =m^2-m\theta,
\label{eqfrtemp}
\end{equation}
complemented with the homogeneous Neumann boundary condition and the initial condition
$\theta(0)=\theta_0.$

\begin{lem}
\label{lemmino2}
Assume that \eqref{initialtemp} holds true.
For fixed $\rho,T>0$, let $m\in \mathcal H_\rho(T)$, then there exists a unique $\theta$ such that
\begin{align*}
&\theta\in C([0,T];H)\cap L^\infty(0,T;L^{3q}(\Omega)\cap V)\cap L^{4q}(0,T;L^{12q}(\Omega))\\
&\theta_t\in L^2(0,T;H)\\
&K(\theta)\in L^\infty(0,T;V)\cap L^2(0,T ;H^2(\Omega))\cap H^1(0,T;L^{\frac 32}(\Omega))\\
&\theta\geq 0, \quad \textrm{a.e. in}\quad Q_T
\end{align*}
which solves
\begin{equation}
\theta_t -\Delta K(\theta) = m^2 - m\theta,\quad  \text{a.e. in}\quad  Q_T
\label{eqfrtempm}
\end{equation}
with $ \pn K(\theta)=0$ a.e. on $\partial \Omega\times (0,T)$ and $\theta(0)=\theta_0,$ a.e. in $\Omega$.
Moreover
\begin{equation}
\label{emme23}
\|\theta\|_{L^4(0,T;L^6(\Omega))}\leq \mu_0(\rho) \mu_1(T),
\end{equation}
for some monotone increasing functions $\mu_i\geq 0$.
\end{lem}
\begin{proof}

Our first concern are formal a priori estimates that can be made rigorous in a Faedo-Galerkin scheme.

\smallskip\noindent
Testing \eqref{eqfrtempm} by $\theta$, we infer
\begin{equation}\label{tetatest}
\12 \dfrac{d}{dt}\Vert \theta\Vert^2 + \io \kappa(\theta)|\nabla \theta|^2 =\io m^2\theta-\io m\theta^2.
\end{equation}
Exploiting the definition of $\kappa(\theta)$ in \eqref{kappa}, the second term on the left hand side reads as
\begin{align}\nonumber
\io \kappa(\theta)|\nabla \theta|^2 &= \Vert\nabla \theta\Vert^2+\io |\theta|^q|\nabla \theta|^2\\ \label{unzipkappa}
&=\Vert \nabla \theta\Vert^2+\dfrac{4}{(q+2)^2}\left\Vert \nabla \left(|\theta|^{\frac{q}{2}+1}\right)\right\Vert^2.
\end{align}
The first term on the right hand side of \eqref{tetatest} can be controlled by H\"{o}lder's and Young's inequalities as
\begin{align}\label{rhs1teta}
\io m^2\theta&\leq \Vert m\Vert_{L^{6}(\Omega)}^2\Vert \theta\Vert_{L^\frac{3}{2}(\Omega)}
\leq c\Vert m\Vert^2_{L^{6}(\Omega)}\Vert\theta\Vert
\leq  c\Vert m\Vert^2_{L^{6}(\Omega)}(1+\Vert\theta\Vert^2).
\end{align}
On the other hand, similarly, we might take care of the second term on the right hand side of \eqref{tetatest}: using the interpolation inequality
$$\Vert \theta\Vert_{L^{\frac{12}5}(\Omega)}\leq c \Vert \theta\Vert^\alpha_{L^2(\Omega)}
\Vert \theta\Vert^{1-\alpha}_{L^6(\Omega)},$$
with $\dfrac{5}{12}=\dfrac{\alpha}{2}+\dfrac{1-\alpha}{6}$ $\left(\alpha=\dfrac 34\right)$, and the Sobolev embedding $W^{1,2}(\Omega)\hookrightarrow L^6(\Omega)$, it holds
\begin{align}
\nonumber
\io m\theta^2 & \leq
c\Vert m\Vert_{L^6(\Omega)}\Vert \theta\Vert^{\frac 32}_{L^2(\Omega)}\Vert \theta\Vert^{\frac12}_{L^6(\Omega)}
\leq
c\Vert m\Vert_{L^6(\Omega)}\Vert \theta\Vert^{\frac 32}\Vert \theta\Vert^{\frac12}_{V}
  \leq \dfrac{1}{2}\Vert \theta\Vert^2_V+c\Vert m\Vert^{\frac43}_{L^6(\Omega)}\Vert\theta\Vert^2\\
  & \leq \dfrac{1}{2}\Vert \nabla \theta\Vert^2+c(1+\Vert m\Vert^{\frac43}_{L^6(\Omega)})\Vert\theta\Vert^2,
\label{rhs2teta}
\end{align}
where we have also exploited Young's inequality.
Combining \eqref{unzipkappa}, \eqref{rhs1teta}, \eqref{rhs2teta} in \eqref{tetatest}, we infer
\begin{align*}
\dfrac{d}{dt}\Vert\theta\Vert^2+\Vert \nabla \theta\Vert^2+\dfrac{2}{(q+2)^2}\left\Vert \nabla \left(|\theta|^{\frac{q}{2}+1}\right)\right\Vert^2
\leq c \left(1+\Vert m\Vert_{L^{6}(\Omega)}^2\right)\Vert\theta\Vert^2+c\Vert m\Vert_{L^{6}(\Omega)}^2,
\end{align*}
hence
\begin{align}
\Vert\theta\Vert_{L^\infty(0,T;H)\cap L^2(0,T;V)}+\left\Vert \nabla \left(|\theta|^{\frac{q}{2}+1}\right)\right\Vert_{L^2(0,T;H)}\leq c(\rho,T),
\label{thinfty}
\end{align}
where, in this proof, $c\geq 0$ stands for monotone increasing function in both its arguments and may possibly change from line to line.
The bound on the second norm above yields
\begin{equation}
|\theta|^{\frac{q}{2}+1}\in L^2(0,T;V),
\end{equation}
since
$$
\||\theta|^{\frac{q}{2}+1}\|^2_V\leq c(\|\theta\|^{q+2}_{L^1(\Omega)}+\|\nabla(|\theta|^{\frac{q}{2}+1})\|^2)
\leq c(\|\theta\|^{q+2}+\|\nabla(|\theta|^{\frac{q}{2}+1})\|^2)
$$
according to the  nonlinear Poincar\'{e} inequality \eqref{nlPoi} and the continuous inclusion $L^2(\Omega)\subset L^1(\Omega)$.

\medskip\noindent
%
\noindent
The test of the temperature equation \eqref{eqfrtemp} by $\theta|\theta|^{3q-2}$ leads to
\begin{align}\label{testthpower}
\io \theta_t \theta|\theta|^{3q-2}+\io\kappa(\theta)\nabla\theta\cdot\nabla(\theta|\theta|^{3q-2})= \io m^2\theta|\theta|^{3q-2}-\io m|\theta|^{3q}.
\end{align}
We first focus on the left hand side of \eqref{testthpower}; since
$$\dfrac{d}{dt}\io|\theta|^{3q}= 3q\io \theta|\theta|^{3q-2} \theta_t, $$
the first term on the left hand side can be equivalently written as
\begin{align}\label{lhsthpower1}
\io \theta_t \theta|\theta|^{3q-2} =\dfrac{1}{3q}\dfrac{d}{dt}\io|\theta|^{3q}.
\end{align}
On the other hand,
$$\nabla (\theta|\theta|^{3q-2})=(3q-1)|\theta|^{3q-2}\nabla\theta$$
entails, using the definition of $\kappa(\th)$ in \eqref{kappa},
\begin{align}\nonumber
\io\kappa(\theta)\nabla\theta\cdot\nabla(\theta|\theta|^q)
&=(3q-1)\io (|\theta|^{3q-2}+|\theta|^{4q-2})|\nabla\theta|^2\\ \label{lhsthpower2}
&=\dfrac{4(3q-1)}{(3q)^2}\Vert \nabla \left(|\theta|^\frac{3q}{2}\right)\Vert^2+\dfrac{3q-1}{4q^2}\Vert\nabla(|\theta|^{2q})\Vert^2.
\end{align}
Inserting \eqref{lhsthpower1} and \eqref{lhsthpower2} in \eqref{testthpower}, we infer
\begin{align}\nonumber
&\dfrac{1}{3q}\dfrac{d}{dt}\io|\theta|^{3q}+\dfrac{4(3q-1)}{(3q)^2}\Vert \nabla \left(|\theta|^\frac{3q}{2}\right)\Vert^2+\dfrac{3q-1}{4q^2}\Vert\nabla(|\theta|^{2q})\Vert^2\\
&\qquad= \io m^2\theta|\theta|^{3q-2}-\io m|\theta|^{3q}=: I_A+I_B.
\label{testthpowerbis}
\end{align}
We now take care of the right hand side.
For simplicity, we start by studying the second term $I_B$.
We preliminarily observe that, by interpolation,
\begin{equation}
\label{interpol3q2}
\|\theta\|_{L^{\frac{3q}2}(\Omega)}\leq \|\theta\|^{\frac 2{3q-2}}_{L^2(\Omega)}\|\theta\|^{\frac {3q-4}{3q-2}}_{L^{3q}(\Omega)}.
\end{equation}

By H\"{o}lder's inequality with exponents $6$, $3/2$ and $6$, the Sobolev embedding $W^{1,2}(\Omega)\hookrightarrow L^6(\Omega)$,
\eqref{interpol3q2} and then Young's inequality together with \eqref{nlPoi}, we have
\begin{align}\nonumber
I_B:=-\io m|\theta|^{3q}&=\io |m||\theta|^q|\theta|^{2q}
\leq c\Vert m\Vert_{L^{6}(\Omega)} \|\theta\|^q _{L^{\frac{3q}2}(\Omega)}\Vert |\theta|^{2q}\Vert_V\\ \nonumber
&\leq \dfrac{\delta}{2}\Vert |\theta|^{2q}\Vert^2_V+ c \Vert m\Vert^2_{L^{6}(\Omega)}\|\theta\|^{2q}_{L^{\frac{3q}2}(\Omega)}\\
\nonumber
&\leq \dfrac{3q-1}{16q^2}\Vert \nabla | \theta|^{2q}\Vert^2+c \Vert \theta\Vert^{4q}_{L^1(\Omega)}+ c \Vert m\Vert^2_{L^{6}(\Omega)}
\|\theta\|^{\frac{4q}{3q-2}}_{L^{2}(\Omega)} \|\theta\|^{2q\frac{3q-4}{3q-2}}_{L^{3q}(\Omega)}
\\ \label{IB}
&\leq \dfrac{3q-1}{16q^2}\Vert \nabla |\theta|^{2q}\Vert^2+c \Vert \theta\Vert^{4q}+ c \Vert m\Vert^2_{L^{6}(\Omega)}\|\theta\|^{3q}_{L^{3q}(\Omega)}+c \Vert m\Vert^2_{L^{6}(\Omega)},
\end{align}
for a suitable choice of $\delta>0$ depending on $q$.
On the other hand, arguing similarly as above, with exponents $3,6$ and $2$ we obtain
\begin{align}\nonumber
I_A:=\io m^2\theta|\theta|^{3q-2} &
\leq
\io m^2 |\theta|^{2q}|\theta|^{q-1}
\leq \Vert m\Vert_{L^6(\Omega)}^2 \Vert |\theta|^{2q}\Vert_V
\|\theta\|^{q-1}_{L^{2q-2}(\Omega)}
\\
\nonumber
&
\leq \dfrac{\delta}{2} \Vert |\theta|^{2q}\Vert^2_V +c\Vert m\Vert_{L^6(\Omega)}^4
\|\theta\|^{2(q-1)}_{L^{2q-2}(\Omega)}\\
\nonumber
&
\leq \dfrac{3q-1}{16q^2}\Vert \nabla |\theta|^{2q}\Vert^2+c \Vert \theta\Vert^{4q}_{L^1(\Omega)} +c\Vert m\Vert_{L^6(\Omega)}^4
\|\theta\|^{3q}_{L^{3q}(\Omega)}+c \Vert m\Vert^4_{L^{6}(\Omega)}\\
&
\leq \dfrac{3q-1}{16q^2}\Vert \nabla |\theta|^{2q}\Vert^2+c \Vert \theta\Vert^{4q} +c\Vert m\Vert_{L^6(\Omega)}^4\|\theta\|^{3q}_{L^{3q}(\Omega)}+c \Vert m\Vert^4_{L^{6}(\Omega)}
\end{align}
having also exploited \eqref{nlPoi} and properly chosen $\delta>0$.
Replacing the above estimates in \eqref{testthpowerbis}, on account of \eqref{thinfty}, we have
\begin{align}\nonumber
&\dfrac{1}{3q}\dfrac{d}{dt}\|\theta\|^{3q}_{L^{3q}(\Omega)}+\dfrac{4(3q-1)}{(3q)^2}\Vert \nabla \left(|\theta|^\frac{3q}{2}\right)\Vert^2+\dfrac{3q-1}{8q^2}\Vert\nabla(|\theta|^{2q})\Vert^2\\
&\qquad\leq  c(\Vert m\Vert_{L^6(\Omega)}^4+1)\|\theta\|^{3q}_{L^{3q}(\Omega)}+c(\Vert m\Vert_{L^6(\Omega)}^4+1)
\label{testthpowerfinal}
\end{align}
for some $c=c(\rho,T)$. Therefore $\theta\in L^\infty(0,T;L^{3q}(\Omega))$ with
\begin{equation}
\label{3q}
\|\theta\|_{L^\infty(0,T;L^{3q}(\Omega))}+\| |\theta|^{2q}\|_{L^2(0,T;V)}\leq \mu_0(\rho)\mu_1(T),
\end{equation}
for some monotone increasing functions $\mu_i\geq 0$, possibly changing from line to line.
In particular, the previous inequality yields
\begin{equation}
\label{boundKappone1}
\|K(\theta)\|_{L^\infty(0,T;H)}\leq  \mu_0(\rho)\mu_1(T).
\end{equation}
since $2q+2\leq 3q$.

\smallskip\noindent
Our next task is an estimate for the time derivative of $\theta$: we then observe that, according to \eqref{kappone}, the temperature equation \eqref{eqfrtemp} reads  as
$$
\theta_t-\Delta K(\theta)=m^2-m\theta,
$$
which we formally multiply by $[K(\theta)]_t=\kappa(\theta)\theta_t$.
Then we infer
\begin{align}\label{test-tetat}
\Vert\sqrt{\kappa(\theta)}\theta_t\Vert^2+\12\dfrac{d}{dt}\Vert\nabla K(\theta)\Vert^2=\io m^2\kappa(\theta)\theta_t-\io m \theta \kappa (\theta)\theta_t=:J_1 +J_2.
\end{align}
The first term on the righthand side is managed by  H\"older's inequality with exponents $3,6$ and $2$, Young's inequality, \eqref{kappa} and \eqref{3q}, so that
\begin{align}\nonumber
J_1:=\io m^2\sqrt{\kappa(\theta)}\sqrt{\kappa(\theta)}\theta_t
&
\leq \|m\|^2_{L^6(\Omega)} \|\sqrt{\kappa(\theta)}\|_{L^6(\Omega)}\|\sqrt{\kappa(\theta)}\theta_t\|
\\
\nonumber
& \leq \dfrac{1}{4}\Vert \sqrt{\kappa(\theta)}\theta_t\Vert^2+c\Vert m\Vert^4_{L^{6}(\Omega)}\left(\io \left[\kappa(\theta)\right]^3\right)^\frac{1}{3}\\ \nonumber
&\leq \dfrac{1}{4}\Vert \sqrt{\kappa(\theta)}\theta_t\Vert^2+c\Vert m\Vert^4_{L^{6}(\Omega)}
\left(\io \left( 1+|\theta|^{3q}\right)\right)^\frac{1}{3}\\ \label{Teta1}
&\leq \dfrac{1}{4}\Vert \sqrt{\kappa(\theta)}\theta_t\Vert^2+c\Vert m\Vert^4_{L^{6}(\Omega)}.
\end{align}
Arguing similarly, we can estimate $J_2$ as follows.
\begin{align}
\nonumber
J_2 & =\io m \theta \kappa(\theta)\theta_t=\io m  (\sqrt{\kappa(\theta)} \theta)\sqrt{\kappa(\theta)}\theta_t \\
\nonumber
& \leq \|m\|_{L^6(\Omega)} \|\theta \sqrt{\kappa(\theta)}\|_{L^3(\Omega)} \|   \|\sqrt{\kappa(\theta)}\theta_t\|\\
\nonumber
& \leq \dfrac 14  \|\sqrt{\kappa(\theta)}\theta_t\|^2+c\|m\|^2_{L^6(\Omega)} \|\theta \sqrt{\kappa(\theta)}\|^2_{L^3(\Omega)}\\
\nonumber
& \leq \dfrac 14  \|\sqrt{\kappa(\theta)}\theta_t\|^2+c\|m\|^2_{L^6(\Omega)}\left(\io (|\theta|^3+|\theta|^{\frac 32q+3} \right)^{\frac23}\\
\nonumber
& \leq \dfrac 14  \|\sqrt{\kappa(\theta)}\theta_t\|^2+c\|m\|^2_{L^6(\Omega)}\left(\io (1+|\theta|^{3q} )\right)^{\frac23}\\
& \leq \dfrac 14  \|\sqrt{\kappa(\theta)}\theta_t\|^2+c\|m\|^2_{L^6(\Omega)},
\label{Teta2}
\end{align}
having applied the H\"older inequality with exponents $6, 3$ and $2$.
Replacing \eqref{Teta1} and \eqref{Teta2} in \eqref{test-tetat}, we are lead to
$$
\Vert\sqrt{\kappa(\theta)}\theta_t\Vert^2+\dfrac{d}{dt}\Vert\nabla K(\theta)\Vert^2\leq c(\|m\|^4_{L^6(\Omega)}+1).
$$
An integration in time on account of \eqref{boundKappone1} provides
\begin{equation}
\label{KappaLinfV}
\Vert \sqrt{\kappa(\theta)}\theta_t\Vert_{L^2(0,T;H)}+\Vert K(\theta)\Vert_{L^\infty(0,T;V)}\leq c(\rho,T).
\end{equation}
Thus, it follows in particular that
\begin{align}
\label{thinftyV}
& \Vert \theta\Vert_{L^\infty(0,T;V)}\leq c(\rho,T),\\
\label{thtbound}
&\Vert \theta_t\Vert_{L^2(0,T;H)}\leq c(\rho,T),\\
\label{Deltathbound}
&\Vert \Delta K(\theta)\Vert_{L^2(0,T;H)}\leq c(\rho,T).
\end{align}
Besides, by interpolation, we eventually infer
\begin{equation}
K(\theta)\in L^4(0,T;W^{1,3}(\Omega)),
\end{equation}
so that, by definition of $K$,
\begin{equation}
\label{key}
\|\nabla \theta\|_{L^4(0,T;L^3(\Omega))}\leq c(\rho,T).
\end{equation}
We can also prove that
$\pt[ K(\theta)]=\kappa(\theta) \theta_t\in L^2(0,T;L^{ \frac 32}(\Omega))$.
Indeed, by \eqref{3q} we learn that
\begin{equation}
\label{capo}
\|\kappa(\theta)\|_{L^\infty(0,T;L^3(\Omega))}\leq c(\rho,T).
\end{equation}
Therefore, owing also to \eqref{KappaLinfV},
$$
\int_0^T \|\pt [ K(\theta(s)]\|^2_{L^{\frac32}(\Omega)}ds\leq \|\kappa(\theta)\|_{L^\infty(0,T;L^3(\Omega))} \|\sqrt{\kappa(\theta)}\theta_t\|_{L^2(Q_T)}\leq c
$$
Since we already know $K(\theta)\in L^2(0,T;H^2(\Omega))$ by \eqref{KappaLinfV} and  \eqref{Deltathbound}, it follows
\begin{equation}
\label{capcont}
K(\theta)\in C([0,T];V).
\end{equation}

\medskip\noindent
The previous computations hold for each approximating solution $\theta_n$ in a Faedo-Galerkin scheme, therefore
\begin{align*}
&\|\pt \theta_n\|_{L^2(Q_T)}+\| \theta_n\|_{L^\infty(0,T,V)}+\|\theta_n\|_{L^\infty(0,T;L^{3q}(\Omega))}
+\|\theta_n\|_{L^4(0,T;W^{1,3}(\Omega))}\leq c(\rho,T)\\
& \\
& \|K(\theta_n)\|_{ L^\infty(0,T;V) \cap L^2(0,T;H^2(\Omega))}+\|\pt K(\theta_n)\|_{L^2(0,T;L^{\frac32}(\Omega))}\leq c(\rho,T)
\end{align*}
as well as
\begin{equation}
\| |\theta_n|^{2q}\|_{L^2(0,T,V)}\leq c(\rho,T),
\end{equation}
where $c(\rho,T)$ does not depend on $n$.
Besides,
\begin{equation}
\label{elle12}
\|\theta_n\|_{L^{4q}(0,T;L^{12q}(\Omega))}\leq c(\rho,T),
\end{equation}
on account \eqref{3q} and of the inclusion $V\subset L^6(\Omega)$.
Indeed,
\begin{align*}
\|\theta_n\|^{4q}_{L^{4q}(0,T;L^{12q}(\Omega))}& =\int_0^T \left(\io |\theta_n|^{12q} \right)^{\frac13}=
\int_0^T \left(\io (|\theta_n|^{2q})^{6} \right)^{\frac13}\\
& \leq
c \int_0^T \left(\||\theta_n|^{2q}\|^2_V \right)\leq c(\rho,T).
\end{align*}
Being
$$
W^{1,3}(\Omega)\Subset L^s(\Omega),\quad \text{for any finite}\quad s\geq 3
$$
and applying Aubin-Simon compactness results,
the inclusion of
$H^1(0,T;L^2(\Omega))\cap L^4(0,T;W^{1,3}(\Omega))$ in $L^4(0,T;L^s(\Omega))$ is compact for any $s\geq 3$.
(In particular, when $s=\frac{12q}{q+1}$).
Then, up to a subsequence.
\begin{equation}
\label{convth1}
\theta_n\to \theta\quad \textrm{strongly in}\quad L^4(0,T;L^s(\Omega)),
\end{equation}
for any finite $s\geq 3$, as well as
\begin{align*}
\pt \theta_n \to \pt \theta\quad \textrm{weakly in}\quad L^2(0,T;L^2(\Omega)),\\
\theta_n\to \theta\quad \textrm{weakly in}\quad L^4(0,T;W^{1,3}(\Omega))\cap L^{4q}(0,T;L^{12q}(\Omega)),\\
\theta_n\to \theta\quad \text{weakly-* in}\quad L^\infty(0,T;L^{3q}(\Omega))\\
\theta_n\to \theta\quad \text{a.e. in}\quad \Omega\times (0,T)).
\end{align*}
Thus, we also have
\begin{align*}
K(\theta_n)\to K(\theta) \quad \textrm{strongly in}\quad C([0,T];H)\quad \text{and}\quad  L^2(0,T;V)\\
 K(\theta_n) \to K(\theta)\quad \textrm{weakly in}\quad H^1(0,T;L^{\frac 32}(\Omega))\quad \text{and} \quad L^2(0,T;H^2(\Omega)),\\
  K(\theta_n) \to K(\theta)\quad \textrm{weakly-* in}\quad L^\infty(0,T;V).
\end{align*}
By these convergences  $\theta$ is indeed a solution to \eqref{eqfrtemp}.
Actually, $\theta\geq 0$ a.e. in $\Omega\times (0,T)$
and therefore there was no loss of generality in $\kappa(s)=1+|s|^q$.

\smallskip\noindent
Indeed, using the Young's inequality on the term $m\theta$ in the temperature equation \eqref{eqfrtemp}, it is easy to see that
\begin{align*}
\theta_t-\dive[\kappa(\theta)\nabla\theta] &=m^2-m\theta\\
&\geq m^2-\dfrac{\theta^2}{2}-\dfrac{m^2}{2}\geq -\dfrac{\theta^2}{2}.
\end{align*}
Hence $v \equiv 0$ is a subsolution to \eqref{eqfrtemp} and thus $\theta(\cdot,t)\geq 0, \quad \forall t\in [0,T]$.

\medskip\noindent
Once the existence of a solution to \eqref{eqfrtemp} is proved, we also show its uniqueness.

\medskip\noindent
Having fixed two solutions $\theta_1,\theta_2$ departing from the same initial datum $\theta_0$,
notice that $\Theta=\theta_1-\theta_2$ solves
\begin{equation}
\label{eqtempdifferenza}
\Theta_t+A [K(\theta_1)-K(\theta_2)]=K(\theta_1)-K(\theta_2)-m\Theta,
\end{equation}
where $A=-\Delta+\mathbb I:V\rightarrow V'$ is the invertible operator with inverse $\mathcal N=A^{-1}$.
Multiplying the above equation by $\mathcal N\Theta$ we are led to
\begin{align*}
\dfrac 12 \ddt \|\Theta\|^2_*+(K(\theta_1)-K(\theta_2),\Theta)=(K(\theta_1)-K(\theta_2),\mathcal N\Theta)-(m\Theta,\mathcal N\Theta).
\end{align*}
Being $(K(\theta_1)-K(\theta_2),\Theta)\geq \|\Theta\|^2$, it follows
\begin{align*}
\dfrac 12 \ddt \|\Theta\|^2_*+\|\Theta\|^2\leq \|K(\theta_1)-K(\theta_2)\|_{L^{\frac65}(\Omega)}\|\mathcal N \Theta\|_{L^6(\Omega)}+\|m\|_{L^3(\Omega)}\|\Theta\|\|\mathcal N \Theta\|_{L^6(\Omega)}.
\end{align*}
We now introduce the function $\ell(r)=r^{q+1}$
for $r\geq 0$, whose derivative is $\ell'(r)=(q+1)r^q $. Taking advantage of this notation, we have
\begin{align*}
|K(\theta_1)-K(\theta_2)|& =\left|\Theta+\dfrac1{q+1} [\ell(\theta_1)-\ell(\theta_2)]\right|\\
& =\left|\Theta+\dfrac 1{q+1}\int_0^1 \ell'(s\theta_1+(1-s)\theta_2)ds\,  \Theta\right|\\
& =|\Theta| \left|1+\int_0^1 |s\theta_1+(1-s)\theta_2|^q ds\, \right|\\
& \leq |\Theta|(1+|\theta_1|^q+|\theta_2|^q).
\end{align*}
Therefore,
\begin{align*}
\|K(\theta_1)-K(\theta_2)\|_{L^{\frac 65}(\Omega)}& \leq c \left(\into |\Theta|^{\frac 65} (1+|\theta_1|^{\frac 65 q}+|\theta_2|^{\frac 65 q})\right)^{5/6}\\
& \leq \|\Theta\| \left(\into (1+|\theta_1|^{\frac 65 q}+|\theta_2|^{\frac 65 q})^{\frac 52}\right)^{\frac13}\\
& \leq c \|\Theta\| \left(\into (1+|\theta_1|^{3 q}+|\theta_2|^{3 q})\right)^{\frac13}\\
&
\leq c\|\Theta\| (1+\|\theta_1\|^{q}_{L^{3q}(\Omega)}+\|\theta_2\|^{q}_{L^{3q}(\Omega)}),
\end{align*}
by H\"older's inequality with exponents $5/3$ and $5/2$. Thus, by the inclusion $W^{1,2}(\Omega)\subset L^6(\Omega)$ and the continuity of $\mathcal N$, we deduce
\begin{align}
\nonumber
\|K(\theta_1)-K(\theta_2)\|_{L^{\frac 65}(\Omega)}\|\mathcal N \Theta\|_{L^6(\Omega)}
& \leq c \|\Theta\|_* \|\Theta\|
(1+\|\theta_1\|^{q}_{L^{3q}(\Omega)}+\|\theta_2\|^{q}_{L^{3q}(\Omega)})\\
& \leq \dfrac 14 \|\Theta\|^2+c\|\Theta\|^2_* (1+\|\theta_1\|^{q}_{L^{3q}(\Omega)}+\|\theta_2\|^{q}_{L^{3q}(\Omega)})^2.
\label{kappadiff1}
\end{align}
Recalling that
$$\|\theta_1\|_{L^\infty(0,T;L^{3q}(\Omega))}+\|\theta_2\|_{L^\infty(0,T;L^{3q}(\Omega))}\leq c,$$
we are lead to
$$
\ddt \|\Theta\|^2_*+\|\Theta\|^2 \leq c(1+\|m\|^2_{L^3(\Omega)})\|\Theta\|^2_*,
$$
therefore Gronwall's Lemma allows to conclude.
\end{proof}

\begin{proof}[Proof of Theorem \ref{localthm}]
Let us define  the operator $\mathcal T:=\mathcal{T}_2\circ \mathcal T_1$, where $(\fhi,\sigma)=:\mathcal T_1(\tilde{\theta}):Q_T\rightarrow \mathbb{R}^2$, with $\tilde{\theta}\in \Theta_R(T)$ and $\theta=:\mathcal{T}_2(\fhi,\sigma):Q_T\rightarrow \mathbb{R}$, with $(\fhi,\sigma)\in \mathcal X_R(T)$.\\
The local (in time) existence of the solution is obtained showing that, for any fixed $R>0$, the Schauder fixed point theorem applies to the operator
by
\begin{align*}
\mathcal T: \Theta_R(T)& \rightarrow L^4(0,T;L^6(\Omega))\\
\tilde \theta& \mapsto \theta
\end{align*}
 provided that $T$ is small enough.\\
 First of all we see that $\mathcal T$ is well defined for any $T,R>0$.
 Indeed, having fixed $(\theta_0,\ph_0,\sigma_0)$ satisfying \eqref{initialtemp}--\eqref{initialnutr} and \eqref{rbound}, given any $\tilde\theta\in \tilde\Theta_R(T)$,
by Lemma \ref{lemmino1}, there exists a unique solution $(\ph,\sigma)\in \mathcal X_R(T)$ to \eqref{eqfrfhi}-\eqref{eqfrnutr} corresponding to $\tilde\theta$, complemented
with homogeneous Neumann boundary conditions and meeting the initial data.
Indeed, $\ph_t\in L^4(0,T;L^6(\Omega))$
with
$$
\|\ph_t\|_{L^4(0,T;L^6(\Omega))}\leq m_0(R)m_1(T).
$$
Then, we can plug $m=\ph_t \in \mathcal H_\rho(T)$ with $\rho=m_0(R)m_1(T)$ in \eqref{eqfrtempm} and, by Lemma \ref{lemmino2}, see that there exists a unique solution $\theta\geq 0$ such that
$$
\|\theta\|_{L^{\infty}(0,T;L^{3q}(\Omega))}\leq \mu_0(R)\mu_1(T).
$$
Therefore the operator $\mathcal T (\tilde \theta)=\theta$ is well defined.
We now prove the following.
\begin{itemize}
\item{\bf The operator $\mathcal T:\Theta_R(T)\rightarrow \Theta_R(T)$ with a suitable choice of $T$} since
$$
\|\theta\|_{L^4(0,T;L^6(\Omega))}\leq T^{\frac14}\|\theta\|_{L^{\infty}(0,T;L^{3q}(\Omega))}\leq  T^{\frac14} c(R,T)\leq R,
$$
provided that $T$ is small enough.
\item{\bf The operator $\mathcal T$ is continuous with respect to the $L^4(0,T;L^6(\Omega))$ norm}
This can be easily seen along the same lines of the uniqueness proof.
\item{\bf The operator $\mathcal T$ is compact} since
$$H^1(0,T;H)\cap L^4(0,T;W^{1,3}(\Omega))\Subset L^4(0,T;L^6(\Omega))$$
thus $\mathcal T(\Theta_R(T))$ is compact in $\Theta_R(T)$.
\end{itemize}
Being the assumptions of Schauder fixed point theorem in place for $T=T(R)>0$ small enough, we deduce the local existence of a solution to our problem.

\smallskip\noindent
Moreover, since the estimates needed to prove uniqueness hold true on any time interval $(0,T)$, the solution is also unique.
Indeed, given $(\theta_i,\fhi_i,\sigma_i)$, both solutions originating from the same initial datum $(\theta_0,\fhi_0,\sigma_0)$, let
 $(\theta,\fhi,\sigma)=(\theta_1-\theta_2,\fhi_1-\fhi_2,\sigma_1-\sigma_2)$.
In the sense we specified in Theorem \ref{localthm}, this triplet solves

\begin{equation*}
\begin{cases}
\theta_t-\Delta[K(\theta_1)-K(\theta_2)]+K(\theta_1)-K(\theta_2)\\
\qquad =K(\theta_1)-K(\theta_2)+\pt\fhi_1^2-\pt\fhi_2^2-\theta_1\pt\fhi_1+\theta_2\pt\fhi_2, \quad \text{a.e. in}\,\,Q_T\\
\fhi_t-\Delta \fhi=-[F'(\fhi_1)-F'(\fhi_2)]+\theta+(\PP\sigma_1-\AA)h(\fhi_1)-(\PP\sigma_2-\AA)h(\fhi_2), \,\, \text{a.e. in}\,\,Q_T\\
\sigma_t-\Delta \sigma+\BB\sigma=-\CC \sigma_1h(\fhi_1)+\CC \sigma_2 h(\fhi_2),\quad \text{in}\quad V', \,\, \text{a.e. in}\, (0,T),
\end{cases}
\end{equation*}
complemented with no flux boundary conditions and null initial condition.
We introduce the functional
$$
\mathcal E(t)=\|\theta(t)\|^2_*+\|\fhi(t)\|^2+\dfrac 12 \|\nabla\fhi(t)\|^2+\|\sigma(t)\|^2,
$$
for which a suitable differential inequality holds, as we see exploiting in particular the continuous operator $\mathcal N:V\rightarrow V'$ as before.
Indeed, adding together the product of the first equation above by $\mathcal N \theta$, the second by $\fhi+\dfrac 12 \fhi_t$ and the third by $\sigma$,
we see that

%
%
%
%
%
\begin{align*}
& \dfrac 12\ddt\mathcal E
+\|\nabla \fhi\|^2+\|\nabla \sigma\|^2+\BB \|\sigma\|^2+\dfrac 12\|\fhi_t\|^2+\|\theta\|^2\\
& \leq c \|\fhi\|^2+\|\theta\| \|\fhi\|+c\|\sigma\|\|\fhi\|-\dfrac 12\io [F'(\fhi_1)-F'(\fhi_2)]\fhi_t+\dfrac 12
\|\theta\| \|\fhi_t\|+c\|\sigma\| \|\fhi_t\|\\
 & \quad+\into [K(\theta_1)-K(\theta_2)]\mathcal N \theta+
\into \pt \ph (\pt \fhi_1+\pt\fhi_2)\mathcal N \theta
-\into \theta \pt\fhi_1\mathcal N\theta-\into \theta_2 \pt\fhi \mathcal N\theta.
\end{align*}
We focus on the five non trivial terms on the righthand side: first, notice that
$$
-\dfrac 12\into[F'(\fhi_1)-F'(\fhi_2)]\fhi_t\leq (1+\|\fhi_1\|^2_{L^6(\Omega)}+\|\fhi_1\|^2_{L^6(\Omega)}) \|\fhi\|_V\|\fhi_t\|\leq c\|\fhi\|_V \|\fhi_t\|.
$$
Arguing exactly as in \eqref{kappadiff1},
\begin{align*}
\into [K(\theta_1)-K(\theta_2)]\mathcal N \theta & \leq \| K(\theta_1)-K(\theta_2)\|_{L^{\frac65}(\Omega)}\|\mathcal N \theta\|_{L^6(\Omega)}\\
& \leq
c \|\theta\| \|\theta\|_*(1+\|\theta_1\|^q_{L^{3q}(\Omega)} +\|\theta_2\|^q_{L^{3q}(\Omega)} )\leq c\|\theta\| \|\theta\|_*.
\end{align*}
Applying H\"older's inequality
\begin{align*}
\into \pt \ph (\pt \fhi_1+\pt\fhi_2)\mathcal N \theta & \leq \|\ph_t\| (\|\pt \fhi_1\|_{L^3(\Omega)}+\|\pt\fhi_2\|_{L^3(\Omega)})\|\mathcal N\theta\|_{L^6(\Omega)}\\
&
\leq c\|\ph_t\| (\|\pt \fhi_1\|_{L^3(\Omega)}+\|\pt\fhi_2\|_{L^3(\Omega)}) \|\theta\|_*.
\end{align*}
Finally,
\begin{align*}
-\into \theta \pt\fhi_1\mathcal N\theta-\into \theta_2 \pt\fhi \mathcal N\theta
&\leq \|\theta\| \|\mathcal N\theta\|_{L^6(\Omega)} \|\pt\fhi_1\|_{L^3(\Omega)}+\|\theta_2\|_{L^3(\Omega)}\|\fhi_t\|\|\mathcal N\theta\|_{L^6(\Omega)}\\
& \leq c \|\pt\fhi_1\|_{L^3(\Omega)} \|\theta\|\|\theta\|_*+ c+\|\theta_2\|_{L^3(\Omega)}\|\fhi_t\|\|\theta\|_*.
\end{align*}
Collecting the above estimates and applying Young's inequality, we obtain
$$
\dfrac 12 \ddt \mathcal E\leq c (\|\pt\fhi_1\|^2_{L^3(\Omega)}+\|\pt\fhi_2\|^2_{L^3(\Omega)}+\|\theta_2\|^2_{L^3(\Omega)})\mathcal E,
$$
therefore Gronwall's lemma allows to conclude.

\end{proof}

\begin{rmk}
\label{segnormk}
Provided that initial datum $\theta_0\geq \underline{\theta}>0$ a.e. in $\Omega$, where $\underline\theta$ is a constant, the same argument, used in
Lemma \ref{lemmino2} to prove that $\theta\geq 0$ a.e. in $Q_T$, allows to show that $\theta\geq \underline \theta$ a.e. in $Q_T$.
\end{rmk}

\subsection{Global in time existence of the solution}\label{secglobal}
We finally observe that the local solution is in fact global.

\begin{thm}
\label{globalthm}
Let $(\theta_0,\ph_0,\sigma_0)$ satisfy assumptions \eqref{initialtemp}--\eqref{initialnutr} and let $T>0$ be an arbitrary final time. Then, there exists a unique solution $(\theta,\ph,\sigma)$ to system \eqref{system} on the whole time interval $(0,T)$. In particular,
$$
\theta\geq 0,\quad 0\leq   \sigma\leq 1 \quad \text{and}\quad \ph\geq 0\quad \text{a.e. in}\quad Q_T.
$$
\end{thm}

\begin{proof}
The local existence proved in Theorem \ref{localthm} is actually global, thanks to further regularities of the solution holding for strictly positive times.
Indeed, assume now that $[0,T)$ is the maximal interval of existence of the solution and let $$(\hat \theta,\hat\ph,\hat\sigma)=\lim_{t\to T^-} (\theta(t),\fhi(t),\sigma(t)).$$

\smallskip\noindent
We first observe that $\hat\sigma\in L^\infty(\Omega)$ with $0\leq \hat \sigma\leq 1$.
Indeed, $\sigma \in C([0,T];L^2(\Omega))$ with $0\leq \sigma\leq 1$ a.e. in $Q_T$ then $\hat\sigma\in L^2(\Omega)$ with $0\leq \hat\sigma\leq 1$
meaning  in particular that $\hat\sigma\in L^\infty(\Omega)$.


\smallskip\noindent
The condition $\hat \fhi\in W^{\frac 53,6}(\Omega)$ follows from $\ph\in C([\delta,T]; W^{\frac 53,6}(\Omega))$, for any $\delta>0$.
Indeed, observe that $\ph$ solves
$$
\ph_t-\Delta\ph=\tilde f,
$$
where now $\tilde f\in L^6(\delta,T;L^6(\Omega))$ for any $\delta\in (0,T)$. Then \cite[Theorem 9.1]{LSU} implies
$$\ph\in L^6(\delta,T;W^{2,6}(\Omega)\cap W^{1,6}(\delta,T;L^6(\Omega))$$
so that $\ph\in C([\delta,T];W^{\frac 53, 6}(\Omega))$, yielding $\hat \fhi\in W^{\frac 53,6}(\Omega)$ and $\hat \ph\geq 0$.\\

\smallskip  \noindent
Finally, we see that $\hat\theta\in L^{3q}(\Omega)\cap V$ with $\hat\theta\geq 0$ a.e. in $\Omega$ and $K(\hat\theta)\in V$.
Lemma \ref{lemmino2} yields $\theta\in C([0,T];H)\cap L^\infty(0,T;L^{3q}(\Omega)\cap V)$ with $\theta\geq 0$ a.e. in $Q_T$.
A subsequent application of Strauss' Lemma (see, e.g.,  \cite{MB})
provides
$\theta\in C_w([0,T];L^{3q}(\Omega)\cap V)$, so that $\hat\theta\in L^{3q}(\Omega)\cap V$
and $\hat\theta\geq 0$ exactly as in \cite{MSJEE}. Besides, $K(\hat\theta)\in V$ owing to  \eqref{capcont}.

We then conclude that $(\hat\theta,\hat\ph,\hat\sigma)$ is an admissible initial datum so that we can extend further the solution, covering any time interval $[0,T]$ in a finite number of steps (see \cite{MSJEE}).

\end{proof}

\section*{Aknowledgements}
The authors S. Gatti and E. Ipocoana were partially supported by the project NOBILI$\textunderscore$PRIN2020
“Mathematics for industry 4.0
(Math4I4) 2020F3NCPX". Moreover, S. Gatti is a member of G.N.A.M.P.A. of Istituto Nazionale di Alta Matematica I.N.d.A.M.
Part of this work was done while S.G. and E.I. were visiting the Laboratoire de Math\'{e}matiques et Applications at the Universit\'e de Poitiers, whose hospitality is gratefully aknowledged. In particular, E.I. was financially supported by LIA-LYSM-AMU-CNRS-ECM-INdAM.
Eventually E.I. acknowledges the funding by the DFG within CRC 1114 “Scaling Cascades in Complex Systems" Project Number 235221301, Project C02 “Interface dynamics: Bridging stochastic and hydrodynamic descriptions".


\end{document}